\documentclass{amsart}

\newtheorem{thm}{Theorem}
\newtheorem{lemma}[thm]{Lemma}

\newcommand{\R}{\mathbb{R}}
\newcommand{\C}{\mathbb{C}}

\newcommand{\N}{\mathbb{N}}
\newcommand{\expe}{{\rm e}}
\newcommand\Ga{\Gamma}
\newcommand{\F}[3]{{}_2F_1\hspace{-0.05cm}\left(
  \genfrac{}{}{0pt}{}{#1}{#2};#3\right)}
\renewcommand{\P}{\mathsf{P}}
\newcommand{\Q}{\mathsf{Q}}
\begin{document}

\title{Fourier series representation of Ferrers function $\P$}
\author{Hans Volkmer}
\address{Hans Volkmer\\
    Department of Mathematical Sciences\\
    University of Wisconsin - Milwaukee\\
    volkmer@uwm.edu
}
\maketitle

In \cite[\S 6]{CPV} we considered a trigonometric expansion of the Ferrers function $\Q$ of the second kind \cite[\S14.13.2]{DLMF}. In this note we consider the analogous expansion of the  Ferrers function $\P$ of the first kind.

\begin{thm}\label{Fourier1}
Let $x\in \C-((-\infty,-1]\cup [1,\infty))$ and $\nu,\mu\in\C$ such that $\nu+\mu\not\in -\N$. Then
\begin{eqnarray*}
\P_\nu^\mu(x)&=& \frac{2^\mu}{i\sqrt\pi}(1-x^2)^{\frac12\mu} \frac{\Ga(\nu+\mu+1)}{\Ga(\nu+\frac32)}\\
&&\times\left(u^{\nu+\mu+1}\F{\mu+\frac12,\nu+\mu+1}{\nu+\frac32}{\frac{u}{v}}
-v^{\nu+\mu+1}\F{\mu+\frac12,\nu+\mu+1}{\nu+\frac32}{\frac{v}{u}}\right),
\end{eqnarray*}
where
\[ u=x+i\sqrt{1-x^2},\quad v=x-i\sqrt{1-x^2}=\frac1u.\]
\end{thm}

If we set $x=\cos\theta$,
$\Re\theta\in(0,\pi)$,
then Theorem \ref{Fourier1} gives
\begin{eqnarray}\label{eq3}
&&\hspace*{1cm}
\P_\nu^\mu(\cos\theta)= \frac{2^\mu}{i\sqrt\pi}(\sin\theta)^\mu\frac{\Ga(\nu+\mu+1)}{\Ga(\nu+\frac32)}\\
&&\times\left(\expe^{i(\nu+\mu+1)\theta}\F{\mu+\frac12,\nu+\mu+1}{\nu+\frac32}{\expe^{2i\theta}}
-\expe^{-i(\nu+\mu+1)\theta}\F{\mu+\frac12,\nu+\mu+1}{\nu+\frac32}{\expe^{-2i\theta}}\right).\nonumber
\end{eqnarray}
If $\theta\in(0,\pi)$ then the arguments $w=\expe^{\pm 2i\theta}$ of the hypergeometric function lie on the unit circle $|w|=1$.
Provided the hypergeometric series converges at $\expe^{\pm 2i\theta}$ we obtain  \cite[(14.13.1)]{DLMF}
\begin{equation}\label{Fourier}
\P_\nu^\mu(\cos\theta)= \frac{2^{\mu+1}}{\sqrt\pi}(\sin\theta)^\mu\sum_{k=0}^\infty \frac{\Ga(\nu+\mu+k+1)}{\Ga(\nu+k+\frac32)}
\frac{(\mu+\frac12)_k}{k!} \sin((\nu+\mu+2k+1)\theta) .
\end{equation}
Regarding the convergence of the series in \eqref{Fourier} we have the following result.

\begin{thm}\label{Fourierconvergence}
Let $\theta\in(0,\pi)$, $\nu,\mu\in\C$ such that $\nu+\mu\in\C\setminus-\N$.\\
(a) If $\Re\mu<0$ then the series in \eqref{Fourier} converges absolutely.\\
(b) If $0\le \Re \mu<\frac12$ then the series in \eqref{Fourier} converges, but, if $\theta\ne\frac12\pi$,
it does not converge absolutely.\\
(c) If $\Re\mu\ge \frac12$ and $\theta\ne \frac12\pi$, then the series in \eqref{Fourier} diverges.
\end{thm}
\begin{proof}
(a) It is known \cite[\S15.2(i)]{DLMF} that the Gauss hypergeometric series $\F{a,b}{c}{w}$ converges absolutely on the unit circle $|w|=1$ if $\Re(c-a-b)>0$. In our case $a=\mu+\frac12$, $b=\nu+\mu+1$, $c=\nu+\frac32$ so
$c-a-b=-2\mu$. If $\Re\mu<0$ it follows that the series in \eqref{Fourier} is the sum of two absolutely convergent series and so is
itself absolutely convergent.\\
(b) Suppose that $0\le\Re\mu<\frac12$.
If $-1<\Re(c-a-b)\le 0$ then the Gauss hypergeometric series converges conditionally at $|w|=1$, $w\ne 1$ \cite[\S15.2(i)]{DLMF}.
It follows that the series in \eqref{Fourier} is the sum of two convergent series and so is itself convergent.
However, it is not true that the sum of two conditionally convergent series is conditionally convergent.
We still have to show that the series in \eqref{Fourier} does not converge absolutely if $\theta\ne \frac12\pi$.
According to \cite[(5.11.12)]{DLMF},
\[ \frac{\Ga(a+z)}{\Ga(b+z)}\sim z^{a-b}\]
as $z\to+\infty$.
Therefore, as $k\to\infty$,
\begin{eqnarray*}
&& \frac{\Ga(\nu+\mu+k+1)}{\Ga(\nu+k+\frac32)}\frac{(\mu+\frac12)_k}{k!}=\\
&&\frac{\Ga(\nu+\mu+k+1)}{\Ga(\nu+k+\frac32)}\frac{\Ga(\mu+\frac12+k)}{\Ga(k+1)\Ga(\mu+\frac12)}
\sim \frac{k^{\mu-\frac12} k^{\mu-\frac12}}{\Ga(\mu+\frac12)}=\frac{k^{2\mu-1}}{\Ga(\mu+\frac12)} .
\end{eqnarray*}
Since $1/\Ga(\mu+\frac12)\ne 0$, there are positive constants $\kappa$ and $K$ such that
\[ \left|\frac{\Ga(\nu+\mu+k+1)}{\Ga(\nu+k+\frac32)}\frac{(\mu+\frac12)_k}{k!}\right|\ge
\frac{\kappa}{k},
\]
for $k\ge K$.
The second part of statement (b) now follows from Lemma \ref{l3}.\\
(c)
Suppose that $\Re\mu\ge \frac12$ and the series in \eqref{Fourier} converges. Then the terms of the series must converge to $0$.
It follows that $\sin((\nu+\mu+2k+1)\theta)\to 0$ as $k\to\infty$. By Lemma \ref{l4} this is impossible unless $\theta=\frac12\pi$.
Therefore, the series in \eqref{Fourier} will diverge for $\theta\ne \frac12\pi$.
\end{proof}

\begin{lemma}\label{l3}
Let $a\in \C$ and
$\theta\in(0,\pi)$. Then
\[ \sum_{k=1}^\infty \frac1k |\sin((a+2k)\theta)|=\infty \]
unless $\sin a=0$ and $\theta=\frac12\pi$.
\end{lemma}
\begin{proof}
If $\theta=\frac12\pi$ then $|\sin((a+2k)\theta)|$ is independent of $k$ and the assertion follows.
Now suppose that $\theta\ne \frac12\pi$. Since $|\sin z|\ge |\sin(\Re z)|$ for all $z\in\C$ it is enough to consider $a\in\R$.
Then $|\sin((a+2k)\theta)|\ge \sin^2((a+2k)\theta)$.
Therefore, it is enough to show that
\[ \sum_{k=1}^\infty \frac1k \sin^2((a+2k)\theta) =\infty .\]
Now
\[ \frac1k\sin^2((a+2k)\theta) =\frac1{2k}-\frac1{2k} \cos(2(a+2k)\theta) .\]
The series $\sum_{k=1}^\infty 1/(2k)$ diverges.  The series
$\sum_{k=1}^\infty
\cos(2(a+2k)\theta)/(2k)$ converges by the Dirichlet test
provided the partial sums $\sum_{k=1}^n \cos(2(a+2k)\theta)$ form a bounded sequence.
This is true for $\theta\in(0,\pi)$, $\theta\ne \frac\pi2$.
\end{proof}

\begin{lemma}\label{l4}
Let $a,b\in\C$ and $\sin b\ne 0$. Then the sequence $\sin(a+bn)$ does not converge to $0$ as $n\to\infty$.
\end{lemma}
\begin{proof}
Suppose that $\sin(a+bn)\to 0$ as $n\to\infty$.
We have
\[ \sin(a+b(n+1))=\cos(a+bn)\sin b+\sin(a+bn)\cos b.\]
If we let $n\to\infty$ and use $\sin b\ne 0$ we obtain that $\cos(a+bn)\to 0$. Since $\cos^2x+\sin^2x=1$ this is a contradiction.
\end{proof}

\end{document}